\documentclass[11pt, leqno]{article}
\usepackage{amsmath,amsfonts,amssymb,wasysym}%,hyperref}
\usepackage[latin1]{inputenc}
\usepackage{shapepar}
\usepackage{graphicx}
\usepackage{color}
\numberwithin{equation}{section}
\newcommand{\R}{{\mathbb R}}
\newcommand{\re}{{\mathbb R}}
\newcommand{\ren}{{\mathbb R}^N}

\newcommand{\be}[1]{\begin{equation}\label{#1}}
\newcommand{\ee}{\end{equation}}

\renewcommand{\)}{\right)}

\newcommand{\prf}{\par\smallskip\noindent{\sl Proof. \/}}
\newcommand{\finprf}{\unskip\null\hfill$\;\square$\vskip 0.3cm}

\newenvironment{proof}{\prf}{\finprf}
\newtheorem{theorem}{Theorem}[section]
\newtheorem{lemma}{Lemma}[section]
\newtheorem{corollary}[theorem]{Corollary}
\newtheorem{proposition}[theorem]{Proposition}
\newtheorem{remark}[theorem]{Remark}
\newtheorem{definition}{Definition}[section]

 \newcommand{\nc}{\normalcolor}
\def\qed{\,\unskip\kern 6pt \penalty 500
\raise -2pt\hbox{\vrule \vbox to8pt{\hrule width 6pt
\vfill\hrule}\vrule}\par}
\definecolor{darkblue}{rgb}{0.05, .05, .65}
\definecolor{darkgreen}{rgb}{0.1, .65, .1}
\definecolor{darkred}{rgb}{0.8,0,0}
\begin{document}
\title{\textbf{ Optimal estimates for Fractional \\ Fast diffusion equations }\\[7mm]}

\author{\Large  Juan Luis V\'azquez and Bruno Volzone} %\footnote{e-mail address:~juanluis.vazquez@uam.es}}
\date{} %%  this cancels date in article format

\maketitle

\begin{abstract}
We obtain a priori estimates with best constants for the solutions of the fractional fast diffusion equation $u_t+(-\Delta)^{\sigma/2}u^m=0$,  posed in the
whole space with $0<\sigma<2$, $0<m\le 1$. The estimates are expressed in terms of convenient norms of the initial data, the preferred norms being the $L^1$-norm and the Marcinkiewicz norm. The estimates contain exact exponents and best constants. We also obtain optimal estimates for the extinction time of the solutions in the range $m$ near 0 where solutions may vanish completely in finite time. Actually, our results apply to equations with a more general nonlinearity. Our main tools are symmetrization techniques and comparison of concentrations. Classical results for $\sigma=2$ are recovered in the limit.

\end{abstract}

\setcounter{page}{1}
%%%%%%%%%%%%%%%%%%%%%%%%%%%%%%%%%%%%%%%%%%%%%%%%%%%%%%%%%%%%%%%%%
\section{Introduction}\label{sec.intro}

Symmetrization is a very ancient geometrical idea that has become nowadays a popular tool of obtaining a priori estimates for the solutions of different partial
differential  equations,  notably  those of elliptic and parabolic type. Symmetrization techniques appear in classical works like \cite{MR0046395, PS1951}. The
application of  Schwarz  symmetrization to obtaining a priori estimates for elliptic problems is already described in
\cite{Wein62} and  \cite{Maz}.  The  standard  elliptic result refers to the solutions of an equation of the form
$$
Lu=f,  \qquad Lu=-\sum_{i,j} \partial_i(a_{ij}\partial_j u)\,,
$$
posed in a bounded domain $\Omega\subseteq \ren$;  the coefficients $\{a_{ij}\}$ are assumed to be bounded, measurable and satisfy the usual ellipticity
condition; finally, we take zero Dirichlet boundary conditions on the boundary $\partial\Omega$. The  classical analysis introduced by Talenti \cite{Talenti1,
Talenti3} leads to pointwise comparison between the symmetrized version (more precisely the spherical decreasing rearrangement)  of the actual solution of the
problem $u(x)$ and the radially symmetric solution $v(|x|)$ of some radially symmetric model problem which is posed in a ball with the same volume as $\Omega$.
Sharp a priori estimates for the solutions are then derived. Extensions of this method to more general problems or related equations have led to a
copious
literature.

For parabolic problems this pointwise comparison fails and the appropriate concept is comparison of concentrations, cf. Bandle \cite{Bandle, Band2} and Vazquez
\cite{Vsym82}. The latter considers the evolution problems of the form
\begin{equation}\label{evol.pbm}
\partial_t u=\Delta A(u), \quad u(0)=u_0,
 \end{equation}
where $A$ a monotone increasing real function and $u_0$ is a suitably given initial datum which is assumed to be integrable. For simplicity the problem was
posed
for $x\in \ren$,  but bounded open sets can be used as spatial domains.

Symmetrization techniques were first applied to PDEs involving fractional Laplacian operators in the paper \cite{BV}, where the linear elliptic case is studied.
In our previous paper  \cite{VazVol} we were able to improve on that progress and combine it with the  parabolic ideas of \cite{Vsym82} to establish the
relevant
comparison theorems based on symmetrization for linear and  nonlinear parabolic equations.  To be specific, we deal with equations of the form
\begin{equation}\label{nolin.parab}
\partial_t u +(-\Delta)^{\sigma/2}A(u)=f, \qquad 0<\sigma<2\,.
\end{equation}
Following the known theory for the standard Laplacian, the nonlinearity $A$ is an increasing real function such that $A(0)=0$, and we accept some extra regularity conditions as needed, like $A$ smooth with $A'(u)>0$ for all $u>0$. The problem is posed in the whole space $\ren$. Special attention is
paid to  cases of the form $A(u)=u^m$ with $m>0$; the equation is then called the Fractional Heat Equation (FHE) when $m=1$, the Fractional Porous Medium Equation (FPME)  if $m>1$, and the Fractional Fast Diffusion Equation (FFDE) if $m<1$. Let us recall that the linear equation \ $\partial_t u
+(-\Delta)^{\sigma/2}u=0$  is a model of so-called anomalous diffusion, a much studied topic in
physics.  The interest in these operators has a long history in Probability since the fractional Laplacian operators of the form $(-\Delta)^{\sigma/2}$  are infinitesimal generators of stable L\'{e}vy processes, see \cite{Applebaum, Bertoin, Valdinoc}. Further motivation and references on the literature are given in \cite{BV2012, VazVol}.

\medskip

\noindent {\sc Main results. } In the present paper we  use the parabolic comparison results of \cite{VazVol} to obtain precise a priori estimates for the
solutions of equation \eqref{nolin.parab}. One of these estimates is the so-called $L^1$ into $L^\infty$ smoothing effect. See the precise result in Theorem \ref{thm.opt}, where we give the precise exponents and best constant $C$ in the decay inequality
\begin{equation}
 \|u(\cdot,t)\|_\infty\le C\,\|u_0\|_1^{\sigma\beta}\,t^{-\alpha}.
\end{equation}
The calculation of best constants in functional inequalities is a topic of continuing interest in the theory of PDEs, both in the elliptic and evolution settings. Classical
references to the calculation of best constants by symmetrization methods are Aubin and Talenti's computation of the best constants in the Sobolev inequality in \cite{Aubin76, Talenti2} and Lieb's HLS inequalities \cite{Lieb83}. Our calculation of a priori estimates with exact exponents and best
constants is closely related to  the sharp decay estimate for solutions of the porous medium/fast diffusion equation in
\cite{Vsym82,JLVSmoothing}. When treating the linear case $A(u)=u$, the estimates  are called
ultra-contractivity, see the book \cite{Davies1} where the importance of best constants is stressed for the applications in Physics, notably in quantum field
theory. This reference also explains the relation between parabolic decay estimates and Nash-Sobolev inequalities.

As a further application of the comparison techniques, optimal estimates with initial data in Marcinkiewicz spaces are obtained in Theorems \ref{Marthe} and
\ref{Marcink}.

We also contribute an interesting extension of the comparison results of \cite{VazVol}, which allows to compare the solutions of two equations with different
nonlinearities $A$ and $\widetilde A$, on the condition that the latter must be a concave function. This leads to extended optimal estimates.

An important critical exponent appears repeatedly in the paper as a lower bound,
\begin{equation}
m_c:=(N-\sigma)/N\,.\label{supFFD}
\end{equation}
Since we are assuming $m>0$ and $0<\sigma<2$, it does not appear in dimension $N=1$ if $\sigma\ge 1$.
Thus, we study the question of deciding the possible extinction of solutions in the range $m<m_c$ in terms of some norm of the initial data, and estimating  the
extinction time. First of all, we construct an explicit extinction solution of the fractional fast diffusion equation in this range of $m$'s, formula
\eqref{solvanish}. Then, we obtain optimal estimates by using comparison based on symmetrization. In this direction we improve significantly the results  of  the previous papers \cite{BV2012} and \cite{pqrv2}, by obtaining optimal estimates on the extinction time for data in Marcinkiewicz spaces, cf. Theorem \ref{extinctiontimetheo}.

Our results in this paper are stable under the limit $\sigma\to 2$, where the standard diffusion case is recovered. See \cite{pqrv2} for details on such limit.

%%%%%%%%%%%%%%%%%%%%%%%%%%%%%%%%%%%%%%%%%%%%%%%%%%%%%%%%%%%%%%%%%%%%%%%%%%%%%%%%%%%%%%%
\section{A short recall of symmetrization and comparison results}

The basic ideas and notations on Schwartz symmetrization are well known. We take from \cite{VazVol} some of the concepts that we will use. Let
$\Omega$ be an open set of $\mathbb{R}^{N}$ or the whole space, and let $f$ be a real measurable function on $\Omega$. We will denote by
$\left\vert
\cdot\right\vert $ the $N$-dimensional Lebesgue measure. We define the distribution function $\mu_{f}$ of $f$ as%
\[
\mu_{f}\left(  k\right)  =\left\vert \left\{  x\in\Omega:\left\vert f\left(
x\right)  \right\vert >k\right\}  \right\vert \text{ , }k\geq0,
\]
and the \emph{decreasing rearrangement} of $f$ as%
\[
f^{\ast}\left(  s\right)  =\sup\left\{ k\geq0:\mu_{f}\left(  k\right)
>s\right\}  \text{ , }s\in\left(  0,\left\vert \Omega\right\vert \right).
\]
Furthermore, if $\omega_{N\text{ }}$ is the measure of the unit ball in $\mathbb{R}^{N}$ and $\Omega^{\#}$ is the ball of $\mathbb{R}^{N}$
centered at
the origin having the same Lebesgue measure as $\Omega,$ we define the
function
\[
f^{\#}\left(  x\right)  =f^{\ast}(\omega_{N}\left\vert x\right\vert
^{N})\text{ \ , }x\in\Omega^{\#},
\]
that will be called the spherical decreasing rearrangement of $f$. From this definition it follows that $f$ is rearranged if and only if
$f=f^{\#}$.

\subsection{Mass concentration}

We will provide estimates of the solutions of our parabolic problems in terms of their integrals. For that purpose, the following definition is
remarkably useful.

\begin{definition}
Let $f,g\in L^{1}_{loc}(\R^{N})$ be two radially symmetric functions on $\R^{N}$. We say that $f$ is less concentrated than $g$, and we write
$f\prec
g$ if for
all $R>0$ we get
\[
\int_{B_{R}(0)}f(x)dx\leq \int_{B_{R}(0)}g(x)dx.
\]
\end{definition}
The partial order relationship $\prec$ is called \emph{comparison of mass concentrations}.
Of course, this definition can be suitably adapted if $f,g$ are radially symmetric and locally integrable functions on a ball $B_{R}$. Besides,
if
$f$
and $g$ are locally integrable on a general open set $\Omega$, we say that $f$ is less concentrated than $g$ and we write again $f\prec g$
simply if
$f^{\#}\prec g^{\#}$,  but this extended definition has no use if $g$ is not rearranged.

The comparison of mass concentrations enjoys a nice equivalent formulation if $f$ and $g$ are rearranged:

\begin{lemma}\label{lemma1}
Let $f,g\in L^{1}(\Omega)$ be two rearranged functions on a ball $\Omega=B_{R}(0)$. Then $f\prec g$ if and only if for every convex
nondecreasing
function
$\Phi:[0,\infty)\rightarrow [0,\infty)$ with $\Phi(0)=0$ we have
\begin{equation}
\int_{\Omega}\Phi(f(x))\,dx\leq \int_{\Omega}\Phi(g(x))\,dx.
\end{equation}
This result still holds if $R=\infty$ and $f,g\in L^{1}_{loc}(\R^{N})$ with $g\rightarrow0$ as $|x|\rightarrow\infty$.
\end{lemma}
From this Lemma it easily follows that if $f\prec g$ and $f,g$ are rearranged, then
\begin{equation}
\|f\|_{p}\leq \|g\|_{p}\quad \forall p\in[1,\infty].
\end{equation}

%%%%%%%%%%%%%%%%%%%%%%%%%%%%%%%%%%%%%%%%%%%%%%%
\subsection{Lorentz spaces}
We review here some basic notions and properties of Lorentz spaces, where the initial data of the parabolic problems will be chosen in some of the
situations we
study in this paper.
Let $\Omega$ be an open subset of $\mathbb{R}^{N}$. As in \cite{JLVSmoothing} for instance, for $p>1$ we define the Marcinkiewicz space (or $L^{p}$-weak space)
$\mathcal{M}_{p}(\Omega)$ as the space of all functions $f\in L^{1}_{loc}(\Omega)$ for which there is a constant $C$ such that
\begin{equation}
\int_{K}|f|dx\leq C|K|^{(p-1)/p}\label{Marcknm1}
\end{equation}
for every subset $K$ of $\Omega$ having finite measure. The smallest $C$ in \eqref{Marcknm1} defines a norm $\|f\|_{\mathcal{M}_{p}}$. It is readily seen
that
if
$f\in L^{p}(\Omega)$, then $f\in \mathcal{M}_{p}(\Omega)$ and
\[
\|f\|_{\mathcal{M}_{p}}\leq\|f\|_{p}.
\]
By this definition, it follows that the function
\[
U_{p}(x)=\frac{C}{|x|^{N/p}}
\]
belongs to $\mathcal{M}_{p}(\Omega)$ but (in general) not to $L^{p}(\Omega)$, and it is more concentrated than any function $f\in L^{p}(\Omega)$ having
$L^p$
norm equal or less than $\|U_{p}\|_{\mathcal{M}_{p}}$.

Actually, one can prove that $f\in \mathcal{M}_{p}(\Omega)$ if and only if there is a constant $C>0$ such that
\begin{equation}
\mu_{f}(t)\leq \left(\frac{C}{t}\right)^{p}\label{Marcknm2}
\end{equation}
for all $t>0$. The smallest $C$ in \eqref{Marcknm2} is shown to be a norm, equivalent to $\|f\|_{\mathcal{M}_{p}}$. \smallskip\\
Now we switch to define the Lorentz spaces. We say
that a real measurable function $f$ on $\Omega$ belongs to the
Lorentz space $L^{p,q}\left(  \Omega\right)$, for $1\leq p,q\leq+\infty$ if the
quantity
\begin{equation}
||f||_{L^{p,q}(\Omega)}=\left\{
\begin{array}
[c]{ll}%
\left(
%TCIMACRO{\dint _{0}^{+\infty}}%
%BeginExpansion
{\displaystyle\int_{0}^{|\Omega|}}
%EndExpansion
\left[  s^{\frac{1}{p}}f^{\ast}(s)\right]  ^{q}\frac{ds}{s}\right)  ^{\frac
{1}{q}} & 0<q<\infty\\
\underset{s\in(0,|\Omega|)}{\sup}s^{\frac{1}{p}}\,f^{\ast}(s) & q=\infty
\end{array}
\right.  \label{def lor-zig}%
\end{equation}
is finite. We remark that for $p>1$, and $q\geq1$, the quantity in
(\ref{def lor-zig}) is a seminorm, but it can be equivalently defined replacing $f^{\ast}\left(
t\right)  $ with $$f^{\ast\ast}\left(  s\right)  =\frac{1}{s}\int_{0}%
^{s}f^{\ast}(\sigma)\, d\sigma,$$
which provides a true norm. In addition, we point out that the $L^{p,q}$-norm, for every $1<p,q\leq+\infty,$ is rearrangement invariant, that is
\[
\left\Vert f\right\Vert _{L^{p,q}\left(  \Omega\right)  }=\Vert f^{\#}%
\Vert_{L^{p,q}\left(  \Omega^{\#}\right)  }.
\]
Besides, we emphasize that $L^{p,p}\left(  \Omega\right)  =L^{p}(\Omega)$,
$L^{p,\infty}\left(  \Omega\right)  =\mathcal{M}_{p}(\Omega)$ and
\begin{equation}
L^{p,q}\left(  \Omega\right)\hookrightarrow \mathcal{M}_{p}(\Omega)\label{embed}
\end{equation}
for all $p,q\geq1$.
For more properties of Lorentz spaces we address the interested reader to \cite{Bennett-Sharpley}.
\nc

%%%%%%%%%%%%%%%%%%%%%%%%%%%%%%%%%%%%%%%%%%%%%%%%%%%%%%%%%%%%%%%%%%%%%%%%%%%%
\subsection{The basic comparison result}
Now we briefly introduce the result in \cite{VazVol} we are going to rely on, which is the concentration comparison result for solutions to the Cauchy
problem for the nonlinear parabolic equation
\begin{equation} \label{eq.1}
\left\{
\begin{array}
[c]{lll}%
u_t+(-\Delta)^{\sigma/2}A(u)=f  &  & x\in\R^{N}\,,t>0%
\\[6pt]
u(x,0)=u_{0}(x) &  & x\in\R^{N}.
\end{array}
\right. %
\end{equation}
where when the nonlinearity $A(u)$ is a nonnegative function, smooth on $\R_{+}$,
with $A(0)=0$ and $A'(u)>0$ for all
$u>0$ (extended anti-symmetrically in the general two-signed theory). In \cite{VazVol} we have obtained that a concentration comparison for solutions
to \eqref{eq.1} holds \emph{only} when the nonlinearity $A$ is a \emph{concave} function, while for \emph{convex} $A$ a remarkable example is
constructed, showing that a failure of concentration comparison occurs (see \cite{VazVol}).

This is the precise statement of the result that will be frequently used in the next sections.

\begin{theorem}\label{Main comparison}
Let $u$ be the nonnegative mild solution to problem \eqref{eq.1} with $0<\sigma<2$, with
initial data $u_0\in L^1(\ren)$, $u_0\ge 0$, right-hand side $f\in L^1(Q)$ where $Q=\R^{N}\times (0,\infty)$, $f\ge 0$, and nonlinearity $A(u)$ given by a
concave function with
$A(0)=0$
and
$A'(u)>0$ for all $u>0$. Let $v$ be the solution of the symmetrized problem
\begin{equation} \label{eqcauchysymm.f}
\left\{
\begin{array}
[c]{lll}%
v_t+(-\Delta)^{\sigma/2}A(v)=f^{\#}(|x|,t)  &  & x\in\R^{N}\,, \ t>0,%
\\[6pt]
v(x,0)=u_{0}^{\#}(x) &  & x\in\R^{N},
\end{array}
\right. %
\end{equation}
where $f^{\#}(|x|,t)$ means the spherical rearrangement of $f(x,t)$ w.r. to $x$ for fixed time $t>0$. Then,
for all $t>0$ we have
\begin{equation}
u^\#(|x|,t)\prec v(|x|,t).\label{conccompa}
\end{equation}
In particular, we have $\|u(\cdot,t)\|_p \le\|v(\cdot,t)\|_p$ for every $t>0$ and every $p\in [1,\infty]$.
\end{theorem}
Moreover, the following corollary justifies a reasonable consequence: if the data of problem \eqref{eq.1} are less concentrated that those of
the symmetrized problem, so are the corresponding solutions.
\begin{corollary}\label{corollarycomp}With the same assumptions of Theorem {\rm \ref{Main comparison}}, suppose that $u$ is the solution to problem
\eqref{eq.1} and
$v$ solves
\begin{equation}\label{eqcauchysymm.f1}
\left\{
\begin{array}
[c]{lll}%
v_t+(-\Delta)^{\sigma/2}A(v)=\widetilde{f}(|x|,t)  &  & x\in\R^{N}\,, \ t>0,%
\\[6pt]
v(x,0)=\widetilde{u}_{0}(x) &  & x\in\R^{N},
\end{array}
\right. %
\end{equation}
where $\widetilde{f}\in L^{1}(Q)$, $\widetilde{u}_{0} \in L^{1}(\R^{N})$ are nonnegative, radially symmetric decreasing functions with respect to $|x|$. If
\[
u_{0}^{\#}(|x|)\prec\widetilde{u}_{0}(|x|),\quad f^{\#}(|x|,t)\prec\widetilde{f}(|x|,t)
\]
for almost all $t>0$, then  the conclusion $u^\#(|x|,t)\prec v(|x|,t)$  still holds.
\end{corollary}

%%%%%%%%%%%%%%%%%%%%%%%%%%%%%%%%%%%%%%%%%%%%%%%%%%%%%%%%%%%%%%%%%%%%%%%%%%%%%%

\section{Comparison between different diffusivities}
\label{sec.comparison.nonlin}

Theorem \ref{Main comparison} and Corollary \ref{corollarycomp} can be extended to a more general situation when problems
\eqref{eq.1}-\eqref{eqcauchysymm.f}, or problems \eqref{eq.1}-\eqref{eqcauchysymm.f1}, display two different diffusivities, on the condition that one is
more concentrated then the other. Actually, the only assumption we require is that one of them needs to be concave. To this aim, we introduce the following
concept of comparison of
nonlinearities.
\begin{definition}
Assume that $A,\widetilde{A}:\R_{+}\rightarrow\R{+}$ are two functions which are smooth in $(0,\infty)$. We write $\widetilde{A}\prec A$ if
\[
\widetilde{A}^{\prime}(s)\leq A^{\prime}(s), \quad\forall s>0.
\]
\end{definition}
In the application to the evolution problem, we may say that $A$ is more diffusive than $\widetilde{A}$. The following Lemma provides a sufficient condition to
reverse the order relationship $\prec$ when passing to the inverse functions of the nonlinearities:
\begin{lemma}\label{inversrel}
Let $A,\widetilde{A}:\R_{+}\rightarrow\R_{+}$ smooth, increasing functions with $A(0)=\widetilde{A}(0)=0$, and let $B=A^{-1}$ and
$\widetilde{B}=\widetilde{A}^{-1}$. If we assume that $\widetilde{A}$ is concave and \ $\widetilde{A}\prec A,$ \ then \
$B\prec\widetilde{B}.$
\end{lemma}

\begin{proof}
We have $\widetilde{A}^{\prime}(s)\leq A^{\prime}(s)$ for every $s>0$ and $\widetilde{A}^{\prime}(s_1)\geq \widetilde{A}^{\prime}(s_2)$ for $0<s_1<s_2$. Pick
some $t_0>0$ and let $A(s_1)=\widetilde{A}(s_2)=t_0$. By the stated properties, we have
\[
A^{\prime}(s_1)\geq\widetilde{A}^{\prime}(s_1)\geq \widetilde{A}^{\prime}(s_2).
\]
But the inverse function theorem gives $B^{\prime}(t_0)=1/A^{\prime}(s_1)$ and $\widetilde{B}^{\prime}(t_0)=1/\widetilde{A}^{\prime}(s_2)$. Therefore,
$B^{\prime}(t_0)\leq \widetilde{B}^{\prime}(t_0)$ and this means $B\prec \widetilde{B}$.
\end{proof}

It is then possible to generalize Theorem 3.4 of \cite{VazVol} in the following way:
\begin{theorem}\label{ellcompdiff}
 Let $v$ be the nonnegative solution of  problem
 \begin{equation} \label{whole}
\left\{
\begin{array}
[c]{lll}%
\left(  -\Delta\right)  ^{\sigma/2}v+  B(v)=f\left(  x\right)   &  & in\text{ }%
\R^{N}\\
&  & \\
v(x)\rightarrow0 &  & as\text{ }|x|\rightarrow\infty,
\end{array}
\right. %
\end{equation}
 posed in $\Omega=\ren$, with nonnegative data $f\in L^1(\ren)$ and nonlinearity given by a strictly increasing function $B:\R_{+}\rightarrow\R_{+}$ which is
 smooth, and superlinear:  $B(t)\geq \varepsilon t$
for
some
$\varepsilon>0$ and all $t\geq0$ and $B(0)=0$ . Assume that $\widetilde{B}:\R_{+}\rightarrow\R_{+}$ is a convex function, satisfying the same assumptions of
$B$,
such that
\[
B\prec\widetilde{B},
\]
and let $V$ be the solution of the radial problem
 \begin{equation} \label{wholesym.}
\left\{
\begin{array}
[c]{lll}%
\left(  -\Delta\right)  ^{\sigma/2}V+  \widetilde{B}(V)=f^{\#}\left(  x\right)   &  & in\text{ }%
\R^{N}\\
&  & \\
V(x)\rightarrow0 &  & as\text{ }|x|\rightarrow\infty,
\end{array}
\right. %
\end{equation}
Then, we have the comparison
\begin{equation*}
v^\#\prec V, \qquad B(v^\#) \prec \widetilde{B}(V).
\end{equation*}
\end{theorem}

\begin{proof}
We keep track of the proof of Theorem 3.4 in \cite{VazVol}, replacing $B$ with  $\widetilde{B}$ in the radial problems.
We have to compare the solution $w$ to problem
\begin{equation}
\left\{
\begin{array}
[c]{lll}%
-\operatorname{div}_{x,y}\left(  y^{1-\sigma}\nabla w\right)  =0 &  & in\text{
}\R^{N}\times(0,+\infty)\\[6pt]
\displaystyle{-\frac{1}{\kappa_{\sigma}}\lim_{y\rightarrow0^{+}}y^{1-\sigma}\,\dfrac{\partial w}{\partial y}(x,y)}+\,B(w(x,0))=f\left(  x\right)   &
&
x\in\R^{N},
\end{array}
\right.  \label{wholeext}%
\end{equation}
with the solution $\psi$ to the problem
\begin{equation}
\left\{
\begin{array}
[c]{lll}%
-\operatorname{div}_{x,y}\left(  y^{1-\sigma}\nabla \psi\right)  =0 &  & in\text{
}\R^{N}\times(0,+\infty)\\[6pt]
\displaystyle{-\frac{1}{\kappa_{\sigma}}\lim_{y\rightarrow0^{+}}y^{1-\sigma}\,\dfrac{\partial \psi}{\partial
y}(x,y)}+\,\widetilde{B}(\psi(x,0))=f^{\#}\left(
x\right)   &
&
x\in\R^{N}.
\end{array}
\right.  \label{wholeextsym}%
\end{equation}
Recall that the traces of $w$, $\psi$ over $\ren\times\left\{0\right\}$ are the solutions $v$, $V$ to problems \eqref{whole}, \eqref{wholesym.}. Using the
change
of
variables
\[
z=\left(  \frac{y}{\sigma}\right)  ^{\sigma},
\]
problems \eqref{wholeext} and \eqref{wholeextsym} become respectively

\begin{equation}
\left\{
\begin{array}
[c]{lll}%
-z^{\nu}\dfrac{\partial^{2}w}{\partial z^{2}}-\Delta_{x}w=0 &  & in\text{
}\R^{N}\times(0,+\infty)\\[6pt]
-\dfrac{\partial w}{\partial z}\left(  x,0\right)  =%
\,\sigma^{\sigma-1}\kappa_{\sigma}\left(f\left(  x\right)-B(w(x,0))\right)  &
&
x\in\R^{N},
\end{array}
\right.  \label{wholeextz}%
\end{equation}
and
\begin{equation}
\left\{
\begin{array}
[c]{lll}%
-z^{\nu}\dfrac{\partial^{2}\psi}{\partial z^{2}}-\Delta_{x}\psi=0 &  & in\text{
}\R^{N}\times(0,+\infty)\\[6pt]
-\dfrac{\partial \psi}{\partial z}\left(  x,0\right)  =%
\,\sigma^{\sigma-1}\kappa_{\sigma}\left(f^{\#}\left(  x\right)-\widetilde{B}(\psi(x,0))\right)  &
&
x\in\R^{N},
\end{array}
\right.  \label{wholeextsymz}%
\end{equation}
where $\nu:=2\left(  \sigma-1\right)  /\sigma.$ \
Then, the problem reduces to prove the concentration comparison between the solutions $w(x,z)$ and $\psi(x,z)$ to \eqref{wholeextz}-\eqref{wholeextsymz}.
We
introduce the function
\begin{equation*}
{Z}(s,z)=\int_{0}^{s}(w^{\ast}(\tau,z)-\psi^{\ast}(\tau,z))d\tau,
\end{equation*}
then using standard symmetrization tools, we get the inequality
\begin{equation*}
-(z^{\nu}{Z}_{zz}+p(  s) {Z}_{ss})\leq0\label{symineq}
\end{equation*}
for a.e. $(s,z)\in \left(  0,\infty \right)
\times\left(  0,+\infty\right)  $,  where $p (s) = N^2\omega_N^{2/N} s^{2- 2/ N}$. Obviously, we have
\begin{equation*}
Z(0,y)={Z}_{s}(\infty,y)=0\label{boundcond}.
\end{equation*}
Concerning the derivative of $Z$ with respect to $z$, due to the boundary conditions contained
in
\eqref{wholeextz}-\eqref{wholeextsymz}, we have
\begin{equation*}\label{Z_yboundary.formula}
{Z}_{z}(s,0)\geq \theta_{\sigma}Y(s,0)
\end{equation*}
where $\theta_{\sigma}:=\sigma^{\sigma-1}\kappa_{\sigma},$ and
$$
Y(s,0)=\int_{0}^{s}B(w^*(\tau,0))-\widetilde{B}(\psi^*(\tau,0))\, d\tau.
$$
Now, since $B\prec\widetilde{B}$, we have
\[
B(w^*(\tau,0))-\widetilde{B}(\psi^*(\tau,0))\leq \widetilde{B}(w^*(\tau,0)-\widetilde{B}(\psi^*(\tau,0))
=\widetilde{B}'(\xi)(w^*(\tau,0))-\psi^*(\tau,0))
\]
where $\xi $ is an intermediate value between $w^*(\tau,0)$
and $\psi^*(\tau,0)$. As $\widetilde{B}$ is convex, $\widetilde
{B}^{\prime}$ is an increasing real function and
$$
B(w^*(\tau,0))-\widetilde{B}(\psi^*(\tau,0))\leq \widetilde{B}'(w^*(\tau,0))(w^*(\tau,0)-\psi^*(\tau,0)).
$$
Now the proof proceeds exactly as in the proof Theorem 3.4 in \cite{VazVol}, up to replacing $B$ with $\widetilde{B}$, and it allows to show that
\[
Z(s,z)\leq0.
\]
Then
\[
\int_{0}^{s}B(v^{*})d\sigma\leq \int_{0}^{s}B(V^{*})d\sigma\leq \int_{0}^{s}\widetilde{B}(V^{*})d\sigma,
\]
and the result is proved.
\end{proof}
Furthermore, a clear generalization of Theorems 3.7-3.8 in \cite{VazVol} with the nonlinearities $B,\widetilde{B}$ still hold.

Using the Crandall-Liggett Theorem in nonlinear semigroup theory, we can use Lemma \ref{inversrel}, Theorem \ref{ellcompdiff}, and get the inspiration from the
proofs of Theorems 5.3-5.4 in \cite{VazVol} to show the following result, generalizing Theorem \ref{Main comparison} and Corollary \ref{corollarycomp}:

\begin{theorem}[Parabolic comparison result with different diffusivities]\label{thm.comp.diff.diff}
Let $u$ be the nonnegative mild solution to problem
\begin{equation*}
\label{eqcauchy.f}
\left\{
\begin{array}
[c]{lll}%
u_t+(-\Delta)^{\sigma/2}A(u)=f  &  & x\in\R^{N}, \ t>0\,,%
\\[6pt]
u(x,0)=u_{0}(x) &  & x\in\R^{N},
\end{array}
\right. %
\end{equation*}
 with $0<\sigma<2$, with
initial data $u_0\in L^1(\ren)$, $u_0\ge 0$, right-hand side $f\in L^1(Q)$, $f\ge 0$, and smooth, strictly increasing nonlinearity $A(u)$ on $\R_+$ with
$A(0)=0$. Assume that $\widetilde{A}$ is a concave nonlinearity on $\R_+$ satisfying the same assumptions of $A$ and let $v$ be the mild solution to
\begin{equation}\label{symmorconc}
\left\{
\begin{array}
[c]{lll}%
v_t+(-\Delta)^{\sigma/2}\widetilde{A}(v)=\widetilde{f}(|x|,t)  &  & x\in\R^{N}\,, \ t>0,%
\\[6pt]
v(x,0)=\widetilde{u}_{0}(x) &  & x\in\R^{N},
\end{array}
\right. %
\end{equation}
where $\widetilde{f}\in L^{1}(Q)$, $\widetilde{u}_{0} \in L^{1}(\R^{N})$ are nonnegative, radially symmetric decreasing functions with respect
to $x$. If moreover
\[
\widetilde{A}\prec A,\quad\,u_{0}^{\#}(|x|)\prec\widetilde{u}_{0}(|x|),\quad f^{\#}(|x|,t)\prec\widetilde{f}(|x|,t),
\]
for almost all $t>0$, then  the conclusion $u^\#(|x|,t)\prec v(|x|,t)$   still holds.
\end{theorem}
\nc

%%%%%%%%%%%%%%%%%%%%%%%%%%%%%%%%%%%%%%%%%%%%%%%%%%%%%%%%%%%%%%%%%%%%%%%%%%%%%%
%%%%%%%%%%%%%%%%%%%%%%%%%%%%%%%%%%%%%%%%%%%%%%%%%%%%%%%%%%%%%%%%%%%%%%%%%%%%%%
\section{Optimal estimates for integrable data}
\label{sec.optimal}

Let us consider the Fractional PME: \
\begin{equation}
u_t+(-\Delta)^{\sigma/2}u^m=0\label{FPME}
\end{equation}
in $\ren$ with $\sigma\in (0,2)$, $N\geq2$. Assuming also that $m$ belongs to the so called \emph{supercritical range}
\begin{equation}
 m>m_c:=\frac{N-\sigma}{N}\,,\label{supFFD}
\end{equation}
the following regularizing effect has been proved in \cite{pqrv}: for every
initial data $u_0\in L^1(\ren)$, $u_0\ge 0$, the solution $u$ to problem
\begin{equation} \label{eqcauchy.3}
\left\{
\begin{array}
[c]{lll}%
u_t+(-\Delta)^{\sigma/2}u^m=0  &  & x\in\R^{N}\,,t>0%
\\[6pt]
u(x,0)=u_{0}(x) &  & x\in\R^{N}
\end{array}
\right. %
\end{equation}
satisfies
\begin{equation}\label{smoothing}
 u(x,t)\le C_{\sigma,m,N} \,\|u_0\|_1^{\sigma\beta}\,t^{-\alpha}\,,
\end{equation}
with the explicit exponents given by
\begin{equation}\label{eq.2}
\beta=\frac{1}{N(m-1)+\sigma}\,, \quad \alpha=N\beta\,
\end{equation}
and a constant $C_{\sigma,m,N}$ that is finite but not determined. The exponents in this formula are sharp, and they can be obtained from dimensional
considerations. This is demonstrated by the construction
of the fundamental solutions, called Barenblatt solutions, that was done \cite{vazBaren}. We briefly recall that the Barenblatt solution with initial
data $u(x,0)=M\,\delta(x)$, $M>0$, has the form
\begin{equation}
U(x,t)=t^{-\alpha}F(\xi), \qquad \xi=x\,t^{-\beta}\label{Baren}
\end{equation}
with $\alpha $ and $\beta$ as before (actually, their value follows from dimensional considerations),
and $F=F_{\sigma,m,N,M}(\xi)$ is  the Barenblatt profile corresponding to exponents $\sigma, m$ and mass $M$.  It is also known that $F>0$ is bounded,
radially symmetric, decreasing in $|\xi|$,  and it decays like a negative power of  $|\xi|$ as $|\xi|\to\infty$. Moreover, the Barenblatt solution is
also smooth, as follows from the regularity results of \cite{pqrv4}.

As a first application of  Theorem \ref{Main comparison} we will calculate the best
constant in this regularizing effect, since the Barenblatt solution plays the role of {\sl worst case.}

\begin{theorem}\label{thm.opt}For every $m\in (m_c,1]$ the optimal constant in \eqref{smoothing} is given by
\begin{equation}
C_{\sigma,m,N}=F_{\sigma,m,N,1}(0)\,,
\end{equation}
where $F_{\sigma,m,1}(\xi)$ is the profile of the Barenblatt solution of the FPME with mass $M=1$. In particular, when $\sigma=1$ and $m=1$ the best
constant is explicit:
\begin{equation}
C_{1,1,N}=\Gamma((N+1)/2)\pi^{-(N+1)/2}.
\end{equation}
\end{theorem}

\noindent {\sl Proof.} The result follows by applying Corollary \ref{corollarycomp}. Indeed, such result implies that the $L^\infty$ norm of a
solution $u$ at time $t>0$ is bounded above by the norm of the solution of any problem with more concentrated radial initial data. At this point, we observe
that
the most  concentrated initial data is the Dirac delta, that produces the Barenblatt solution. Such solution with mass
$M=\|u_{0}\|_{1}$ is given by \eqref{Baren}, where $F_{\sigma,m,M}$ is the profile with mass $M$. By the scaling properties shown in \cite{vazBaren},
formula (8.3), we have
\begin{equation}
U_{\sigma,m,M}(x,t)=t^{-\alpha}\,M^{\sigma\beta}F_{\sigma,m,N,1}\left(M^{(1-m)\beta}\,\frac{x}{t^{\beta}}\right),\label{U}
\end{equation}
so that
$$
\|U_{\sigma,m,M}(\cdot,t)\|_{\infty}=t^{-\alpha}M^{\sigma\beta}F_{\sigma,m,N,1}(0).
$$
There is a problem in justifying the previous argument, since the Dirac delta is not a function. We need an approximation argument that is not
difficult and can be seen in \cite{JLVSmoothing} applied to the standard PME.

For the constant of the linear case $m=1$, with the choice $\sigma=1$ we recall that the fundamental solution if explicit,
$U_{1,1,1}(x,t)=t^{-N}F_{1,1,N,1}(x/t)$ with profile
$$
F_{1,1,N,1}(\xi)=C_N\,(1+|\xi|^2)^{-(N+1)/2}, \quad C_N=\Gamma((N+1)/2)\pi^{-(N+1)/2}.
$$
\qed

\medskip

As a second application of the above comparison theory, we get a general $L^{p}$ smoothing effect, for
all $1\leq p<\infty$.

\begin{theorem}\label{thm.estimate.Lp} Let $m\in (m_c,1]$ and choose $1< p<\infty$ and a nonnegative data $u_{0}\in L^{1}(\ren)\cap L^{\infty}(\ren)$. Then
$u(t)\in L^{p}(\ren)$
for all
$t>0$ and there is a constant $C=C(\sigma,m,N,p)$ such that
\begin{equation}
\|u(t)\|_{p}\leq C \|u_{0}\|_{1}^{\frac{\beta}{p}(N(m-1)+\sigma p)\nc} t^{-\alpha(p-1)/p}.
\end{equation}
The constant $C$ is attained by the Barenblatt profile.
\end{theorem}

\begin{proof}
First we observe that by the boundedness of the profile $F=F_{\sigma,m,N,M}$  and its asymptotic behavior we have $F_{\sigma,m,N,1}(\xi)\in L^{p}(\ren)$ for all
$p\in [1,\infty]$. In fact, according to \cite{vazBaren} $F_{\sigma,m,1}(\xi)$ decreases as $|\xi|\to\infty$ like $O(|x|^{-\gamma})$ with
$$
\gamma= \frac{\sigma}{1-m} \quad \mbox{if \ } m\in (m_c,m_1), \qquad \gamma= N+\sigma \quad \mbox{if \ }  m\in (m_1,1),
$$
where $m_1=N/(N+\sigma)$. Then, Corollary \ref{corollarycomp} and formula \eqref{U} imply
\[
\|u(t)\|_{p}\leq\|U_{\sigma,m,M}(t)\|_{p}=t^{-\alpha(p-1)/p}\,\|u_{0}\|_{1}^{\frac{\beta}{p}(N(m-1)+\sigma p)\nc}\,\|F_{m,1}\|_{p}
\]
and the result follows.
\end{proof}

These results can be extended to the more general equation with nonlinearity $A$ not necessarily a power function, by using the comparison result for different
diffusivities, that is Theorem \ref{thm.comp.diff.diff}. We have

\begin{corollary}\label{estim.nonpower} Let $u\ge 0$ be a solution of equation $\partial_t u +(-\Delta)^{\sigma/2}A(u)=0 $ with initial data $u_0\in L^1(\ren)$.
Let $A$ be a $C^1$
function $\re_+\to \re_+$ with $A(0)=0$ and $A'(u)\ge au^{m-1}$ for some $a>0$, $1\geq m>m_c$,and all $u>0$. Then,
\begin{equation}
u(x,t)\le C_{\sigma,m,N,a} \,\|u_0\|_1^{\sigma\beta}\,t^{-\alpha}\,
\end{equation}
where $\alpha$ and $\beta$ are as before and $C_{\sigma,m,N,a}=(m/a)^\alpha\, F_{\sigma,m,N,1}(0)$.
\end{corollary}

\begin{proof}
Let us define the function $\widetilde{A}(u)=(a/m)u^{m}$ for all $u\geq0$: then by assumption we have $A\succ\widetilde{A}$, being $\widetilde{A}$ concave. Let
us take then the fundamental solution $v$, with mass $M:=\|u_0\|_1$ to the equation
\[
v_{t}+(-\Delta)^{\sigma/2}\widetilde{A}(v)=0.
\]
Then we have
\[
v(x,t)=U_{\sigma,m,M}\left(x,\frac{a}{m}t\right),
\]
thus the application of Theorem \ref{thm.comp.diff.diff} (up to a suitable approximation of the Dirac delta) gives the result.
\end{proof}

A similar statement follows from Theorem \ref{thm.estimate.Lp}. We leave the detail to the reader.

%%%%%%%%%%%%%%%%%%%%%%%%%%%%%%%%%%%%%%%%%%%%%%%%%%%%%%%%%%%%%%%%%%%%%%%%%%%%%
\section{Estimates for data in Lorentz spaces}

Assume we are still in  the Fast Diffusion regime, i.\,e.,
\begin{equation*}
u_t+(-\Delta)^{\sigma/2}u^m=0, \qquad 0<m<1.
\end{equation*}
Consider now that the initial data are chosen to belong to the
\emph{Marcinkiewicz space} $\mathcal{M}^{p}(\R^{N})$, $p>1$. In that case, for any given Marcinkiewicz norm $M>0$ we may find the \emph{worst} initial data in
the sense of concentration comparison. The answer is: the function that represents the most concentrated data amongst the ones having the same Marcinkiewicz
norm
has the form $u_0(x)= C\,|x|^{-N/p}$, where $C$ is proportional to $M$. Since such data are not in the standard
class
of integrable data used up to now,
we need to employ Theorem 3.1 in \cite{BV2012} to obtain the existence of a unique minimal very weak solution to problem \eqref{eqcauchy.3} with  data $u_{0}$
in
the Marcinkiewicz space. Indeed, we may choose a weight $\varphi$ defined as
\[
\varphi(x)=(1+|x|)^{-\alpha}
\]
with
\[
N<\alpha<N+\frac{\sigma}{m}.
\]
Then the exponent $\alpha$ satisfies the assumption of Theorem 2.2 in \cite{BV2012} and $u_{0}\in L^{1}(\R^{N},\varphi)$, in the sense
that
\[
\int_{\R^{N}}u_{0}(x)\,\varphi(x)dx<\infty.
\]
Then Theorem 3.1 of \cite{BV2012} ensures  the existence of a very weak solution $u(\cdot,t)\in L^{1}(\R^{N},\varphi)$ to equation
\eqref{FPME} in $\R^{N}\times[0,T]$, in the sense that
\[
\int_{0}^{\infty}\int_{\ren}u\frac{\partial\psi}{\partial t}\,dx\,dt-\int_{0}^{\infty}\int_{\ren}u^{m}(-\Delta)^{\sigma/2}\psi\, dxdt=0\
\label{idveryweak}
\]
for all the test functions $\psi\in C_{0}^{\infty}(\ren\times[0,T])$, $T>0$. Moreover, we have that $u$ is continuous in the weighted space, in the sense that
$u\in C([0,T];L^{1}(\R^{N},\varphi\, dx))$. The solution $u$ is called \emph{minimal solution}, because it is constructed as a monotone limit of the sequence
$\left\{u_{n}\right\}$ of
solutions to problem \eqref{eqcauchy.3} with data $u_{0,n}$, with $u_{0,n}\nearrow u_{0}$ and $u_{0,n}\in L^{1}(\ren)\cap L^{\infty}(\ren)$.
According to Theorem 3.2 of \cite{BV2012}, the minimal solution $u$ to \eqref{eqcauchy.3} is unique.

Once these questions are settled, we are able to give an extension of Theorem \ref{Main comparison} for singular data in Marcinkiewicz spaces.

\begin{theorem}[Comparison result with singular data]\label{Marthe}
Assume that $u_{0}\in \mathcal{M}^{p}(\ren)$ with  $p>1$, $m\in (0,1)$ and let $u$ be the unique minimal solution to problem \eqref{eqcauchy.3}. If $v$ is the
minimal solution to the problem
\begin{equation} \label{symmMarc}
\left\{
\begin{array}
[c]{lll}%
v_t+(-\Delta)^{\sigma/2}v^{m}=0  &  & x\in\R^{N}\,, \ t>0,%
\\[6pt]
v(x,0)=\dfrac{\|u_{0}\|_{\mathcal{M}_{p}}}{|x|^{N/p}} &  & x\in\R^{N},
\end{array}
\right. %
\end{equation}
then for all $t>0$ we have
\[
u^\#(|x|,t)\prec v(|x|,t).\label{conccomp}
\]
\end{theorem}

\begin{proof}
The proof is divided in two steps. \smallskip\\
\noindent {\sc Step 1.} Let us choose a sequence $u_{0,n}\in L^{1}(\ren)\cap L^{\infty}(\ren)$ such that $u_{0,n}\nearrow u_{0}$.
Let $u_{n}$ be the solution to problem \eqref{eqcauchy.3} with data $u_{0,n}$, so that by the comparison principle
\[
u_{n}\nearrow u.
\]
By the properties of rearrangements we have
\[
u_{n}^{\#}\nearrow u^{\#}
\]
in $\ren\times [0,\infty)$. Assume that $\widetilde{v_{n}}$ is the solution to the problem
\begin{equation*}
\left\{
\begin{array}
[c]{lll}%
v_t+(-\Delta)^{\sigma/2}v^{m}=0  &  & x\in\R^{N}\,, \ t>0,%
\\[6pt]
v(x,0)=u_{0,n}^{\#}(x) &  & x\in\R^{N}.
\end{array}
\right. %
\end{equation*}
Then the function $\widetilde{v}$ defined as the monotone limit of $\widetilde{v_{n}}$, that is
\[
\widetilde{v_{n}}\nearrow \widetilde{v}
\]
is the unique minimal solution to the problem
\begin{equation}\label{verysym}
\left\{
\begin{array}
[c]{lll}%
v_t+(-\Delta)^{\sigma/2}v^{m}=0  &  & x\in\R^{N}\,, \ t>0,%
\\[6pt]
v(x,0)=u_{0}^{\#}(x) &  & x\in\R^{N}.
\end{array}
\right. %
\end{equation}
By Theorem \ref{Main comparison} we get
\[
u_{n}^\#(|x|,t)\prec \widetilde{v_{n}}(|x|,t),
\]
so Lebesgue's monotone convergence theorem allows to pass to the limit and get
\begin{equation}
u^\#(|x|,t)\prec \widetilde{v}(|x|,t).\label{intermcomp}
\end{equation}\smallskip
\noindent {\sc Step 2.} Since
\[
u_{0}^{\#}(x)\leq\frac{\|u_{0}\|_{\mathcal{M}_{p}}}{|x|^{N/p}}
\]
we obtain by the comparison principle (see Theorem 3.2 of \cite{BV2012})
\[
\widetilde{v}(|x|,t)\leq v(|x|,t)
\]
and by \eqref{intermcomp} the result follows.
\end{proof}

\begin{remark}\label{extgenA}
\emph{In the same spirit of Corollary \ref{estim.nonpower}, we could extend the estimate of Theorem \ref{Marthe} to the solutions of an equation of the form
\begin{equation}
u_t + (-\Delta)^{\sigma/2} A(u)=0\label{gennonl}
\end{equation}
when $A$ is not a power function. Indeed, assume that $A'(u)\ge au^{m-1}$ for some $a>0$, $1\geq m>m_c$, and all $u>0$ and choose the nonnegative datum $u_0\in
L^{1}(\ren)\cap L^{p}(\ren)$ for $p>\max\left\{1,\widetilde{p}\right\}$ where
\begin{equation}
\widetilde{p}=\frac{N(1-m)}{\sigma}\label{critcexp}
\end{equation}
is the so called \emph{critical exponent}. Then an existence result of the forthcoming paper \cite{pqrv4} assures the existence of a unique weak solution $u$ (in a suitable sense) to
equation \eqref{gennonl} with data $u_0$, which is also a classical solution. Then Theorem \ref{thm.comp.diff.diff} gives
\[
u^{\#}\left(|x|,t\right)\prec \widetilde{v}\left(|x|,\frac{a}{m}t\right)\,,
\]
where $\widetilde{v}$ solves \eqref{verysym}, while
\[
\widetilde{v}(|x|,t)\leq v(|x|,t)\,,
\]
where $v$ solves \eqref{symmMarc},
thus}
\begin{equation}\label{estimgenA}
u^{\#}\left(|x|,t\right)\prec v\left(|x|,\frac{a}{m}t\right).
\end{equation}
\end{remark}

%%%%%%%%%%%%%%%%%%%%%%%%%%%%%%%%%%%%%%%%%%%%%%%%%%%%%%%%%%%%%%%%%%%%%
A really interesting feature of the solutions to problem  \eqref{symmMarc} is that they exhibit the property of self similarity, as the following result shows.

\begin{proposition} \label{prop.marc} With same $\sigma, m, N$ as before, let $U_p$ be  the unique, minimal very weak solution to \eqref{eqcauchy.3} with data
$u_{0}(x)=|x|^{-N/p}$, $p>\max\left\{1,\widetilde{p}\right\}$, where $\widetilde{p}$ is given in \eqref{critcexp}. Then $U_p$ has the form
$$
U_p(x,t)=t^{-\alpha_{p}}F_p(x\,t^{-\beta_{p}})
$$
with  exponents $\beta_{p}$, $\alpha_{p}$ defined by
\begin{equation}
\beta_{p}=\frac{p}{N(m-1)+\sigma p},\,\,\alpha_{p}=\frac{N}{p}\beta_{p}.\label{expLorentz}
\end{equation}
Moreover, the profile $F_p(\xi)$ is  bounded, radially symmetric, and it  behaves like $|\xi|^{-N/p}$ as $|\xi|\rightarrow\infty$.
\end{proposition}

\begin{proof} We follow the ideas of \cite{JLVSmoothing}.
Let us consider the function
\[
\overline{U}_{p}(x,t):=A U_{p}(Bx,Ct).
\]
for some positive constant $\lambda,\,A,\,B$. Then $\overline{U}_p$ solves \eqref{eqcauchy.3} with the same initial data $u_{0}(x)=|x|^{-N/p}$ if and
only if $A,\,B,\,C$ are linked by means of the relations
\[
A=B^{N/p},\quad A^{m-1} B^{\sigma}=C.
\]
Then for any fixed $C>0$, choosing the constants $A,\,B$ according to the equations above, we can write
\[
\overline{U}_{p}(x,t)=C^{\alpha_{p}}U_{p}(C^{\beta_{p}}x,Ct),
\]
with the exponents $\beta_{p}$, $\alpha_{p}$ defined by relations \eqref{expLorentz}.
Now, let us prove that $\overline{U}_{p}$ is another solution obtained by approximating the same data $u_{0}(x)=|x|^{-N/p}$. To this aim, assume that
$u_{0,n}\in L^{1}(\ren)\cap L^{\infty}(\ren)$ is such that $u_{0,n}\nearrow u_{0}$ and let $u_{n}$ be the sequence of solutions to problem
\eqref{eqcauchy.3} with data $u_{0,n}$. For what we explained before, we have
\[
u_{n}\nearrow U_{p}.
\]
Now define the sequence
\[
\overline{u}_{0,n}=C^{\alpha_{p}}\,u_{0,n}(C^{\beta_{p}}x).
\]
Clearly
\[
\overline{u}_{0,n}\nearrow u_{0}.
\]
Then we point out that the sequence of solutions $\overline{u}_{n}$ to \eqref{eqcauchy.3} with data $\overline{u}_{0,n}$ are exactly the functions
\[
\overline{u}_{n}(x,t):=C^{\alpha_{p}}\,u_{n}(C^{\beta_{p}}x,Ct)
\]
and that
\[
\lim_{n\rightarrow\infty}\overline{u}_{n}(x,t)=C^{\alpha_{p}}U_{p}(C^{\beta_{p}}x,Ct)=\overline{U}_{p}(x,t).
\]
By the uniqueness of the minimal solution (Theorem 3.2 of \cite{BV2012}) we find
\[
U_{p}(x,t)=\overline{U}_{p}(x,t)=C^{\alpha_{p}}U_{p}(C^{\beta_{p}}x,Ct).
\]
Therefore, following \cite{JLVSmoothing}, with a suitable choice of the constant $C$ we find that the solution $U_{p}(x,t)$ takes the selfsimilar form
\[
U_{p}(x,t)=t^{-\alpha_{p}}F_{p}(x\,t^{-\beta_{p}}).
\]
Since the data $u_0$ is radially decreasing, the profile $F_{p}(\xi)$ inherits the same property (see \cite{vazBaren}); moreover $F_{p}(\xi)$ decreases like
$|\xi|^{-N/p}$ as $|\xi|\rightarrow\infty$ (see \cite{BV2012}) .
Now in order to prove that $F_{p}(\xi)$ is bounded, it will be enough to notice that $F_{p}(\xi)$ satisfies the fractional elliptic equation
\begin{equation}
(-\Delta)^{\sigma/2}F^{m}=\alpha_{p}F+\beta_{p}\,\xi\cdot\nabla F\label{fractellip}
\end{equation}
and that $F\in L^{q}(B_{1}(0))$ for some $q\in (\widetilde{p},p)$, because $U_{p}(\cdot,t)\in \mathcal{M}_{p}(\ren)$ for all $t\geq0$. Then the argument
will
be achieved by the proposition that follows.\end{proof}
\begin{proposition}
Let $p>\max\left\{1,\widetilde{p}\right\}$ and assume that $F=F(|\xi|)$ is any radially decreasing solution to equation \eqref{fractellip}
and
that $F\in L^{q}(B_{1}(0))$ for $q\in (\widetilde{p},p)$.
Then $F$ is bounded.
\end{proposition}
\begin{proof}
The proof is based upon an application of classical Moser's iteration techniques.
Let us multiply equation \eqref{fractellip} by the the test function $F^{q-1}$ and integrate over the ball $B_{1}(0)$. Integrating by parts,
we find
\begin{align}
&\int_{B_{1}(0)}F^{q-1}(-\Delta)^{\sigma/2}F^{m}d\xi=\alpha_{p}\int_{B_{1}(0)}F^{q}d\xi-\frac{N\beta_{p}}{q}\int_{B_{1}(0)}F^{q}d\xi+\frac{\beta_{p}}
{q}N\omega_{N}F^{q}(1)\label{ellfraceq1}\\
&\leq \alpha_{p}\int_{B_{1}(0)}F^{q}d\xi+\frac{\beta_{p}}
{q}N\omega_{N}F^{q}(1).\nonumber
\end{align}
Note that by the radial monotonicity of $F$ we also have
\[
\int_{B_{1}(0)}F^{q}d\xi\geq\omega_{N}F^{q}(1)
\]
and inserting this inequality in \eqref{ellfraceq1} we find
\begin{equation}
\int_{B_{1}(0)}F^{q-1}(-\Delta)^{\sigma/2}F^{m}d\xi\leq C\int_{B_{1}(0)}F^{q}d\xi.\label{ellfraceq2}
\end{equation}
for some positive constant $C$.
Now it is time to use heavy machinery: indeed, according to Theorem 8.2 in \cite{pqrv2}, applying the bounded versions of
\emph{Strook-Varopoulos's} and \emph{Nash-Gagliardo-Nirenberg type} inequalities (see Lemma 5.1 and 5.3 in \cite{pqrv2}) we obtain
\begin{equation}
C\|F\|^{2q}_{q}\geq\frac{4m(q-1)}{(q-1+m)^2}\|F\|^{2q+m-1}_{r}\label{ellfraceq3}
\end{equation}
where
\[
r=\frac{N}{2N-\sigma}(2q+m-1).
\]
Now we start the iterations. We set $q_{0}=q$ and
\[
q_{k+1}=\rho\(q_{k}+\frac{m-1}{2}\),\quad \rho=\frac{2N}{2N-\sigma}.
\]
We notice that, by our assumptions, the sequence $\left\{q_{k}\right\}$ is defined through
\[
q_{k}=A(\rho^{k}-1)+q\,\quad A=q-\frac{(1-m)N}{\sigma}>0
\]
so that $q_{k}>q_{k+1}$ and $\lim_{k\rightarrow\infty}q_{k}=+\infty$. Then inequality \eqref{ellfraceq3} provides
\begin{equation*}
\frac{4mq_{k}}{(q_{k-1}-1+m)^2}\|F\|^{2q_{k}+m-1}_{q_{k+1}}\leq N\beta \|F\|^{2q_{k}}_{q_{k}}
\end{equation*}
from which
\begin{equation}
\|F\|_{q_{k+1}}\leq \mathsf{L}^{\frac{\rho}{2q_{k+1}}}\|F\|_{q_{k}}^{\frac{\rho
q_{k}}{q_{k+1}}}\label{ellfraceq4}
\end{equation}
for some positive constant $\mathsf{L}$. Then, if we set
\[
U_{k}=\|F\|_{q_{k}}
\]
for $k\geq0$, inequality \eqref{ellfraceq4} yields
\begin{equation}
U_{k}\leq L^{\theta_{k}} U_{0}^{\nu_{k}}\label{ellfraceq5}
\end{equation}
where the exponents converge to the following limits, as $k\rightarrow\infty$:
$$
\theta=\frac{1}{2q_{k}}\sum_{j=1}^{k}\rho^{j}\rightarrow \frac{N}{A},\quad\nu_{k}=\frac{\rho^{k}q}{q_{k}}\rightarrow\frac{q}{A}.$$
Then passing to the limit in \eqref{ellfraceq5} one finally finds
\[
\|F\|_{\infty}=\lim_{k\rightarrow\infty}U_{k}\leq \mathsf{C}_{1} U_{0}^{\frac{q}{A}}=\mathsf{C}_{1}\|F\|_{q}^{\frac{\sigma q}{\sigma q-(1-m)N}}.
\]
\end{proof}
Now we wish to use the scaling argument as in \cite{JLVSmoothing} to better highlight the dependence of such self similar profile on the
Marcinkiewicz
norm $M$, that is when we take data of
the form $u_{0}=M/|x|^{N/p}$. We know that the self-similar profile $F_{p}$ of the solution $U_{p}$ constructed in Proposition \eqref{prop.marc} must satisfy
the
fractional elliptic equation \eqref{fractellip}.
This implies that the rescaled profile
\[
\hat{F}_{p}(\xi)=\lambda^{\mu}F_{p}(\lambda \xi)
\]
satisfies the same equation if
\begin{equation}
\mu=\frac{\sigma}{1-m}.\label{mu}
\end{equation}
Now let us consider the function
\begin{equation}
\hat{U}_{p}(x,t)=t^{-\alpha_{p}}\hat{F}_{p}(x\,t^{-\beta_{p}})=\lambda^{\mu} U_{p}(\lambda x,t).\label{selfMarc}
\end{equation}
Then it is readily seen that $\hat{U}_{p}$ satisfies the equation \eqref{FPME} and takes the initial data
\begin{equation}
u_{0}(x)=\frac{M}{|x|^{N/p}}\label{singdata}
\end{equation}
if and only if
\[
\lambda^{\mu-\frac{N}{p}}=M
\]
\emph{i.e.} if
\begin{equation}
\lambda=M^{(1-m)\beta_{p}}.\label{lamb}
\end{equation}
Besides with the choices \eqref{mu}, \eqref{lamb} of $\mu,\,\lambda$, the self-similar solution \eqref{selfMarc} is actually the minimal solution to
problem
\eqref{eqcauchy.3}
with the singular data \eqref{singdata}. Indeed, choosing the usual approximating sequence $u_{0,n}\nearrow |x|^{-N/p}$ and the corresponding
solutions $u_{n}$ such that
\[
u_{n}\nearrow U_{p}
\]
we define the initial data
\[
\hat{u}_{0,n}:=\lambda^{\mu} u_{0,n}(\lambda x)\nearrow M |x|^{-N/p},
\]
then the functions
\[
\hat{u}_{n}(x,t):=\lambda^{\mu} u_{n}(\lambda x,t)
\]
solve \eqref{FPME} taking the  data $\hat{u}_{0,n}$ and
\[
\lim_{n\rightarrow \infty}\hat{u}_{n}(x,t)=\lambda^{\mu} U_{p}(x,t)=\hat{U}_{p}(x,t).
\]

%%%%%%%%%%%%%%%%%%%%%%%%%%%%%%%%%%%%%%%%%%%%%%%%%%%%%%%%%%%%%%%%%%%%%%%%%%%%%%%%%%%%%%%%%%%%%%%%%%%%%%%%%%%%%%%%%%%%%%%%%%%%%%%%%%%%%%%%%%%%%%%%%%

All this information allows to derive the following result, extending the smoothing effect in the Marcinkiewicz spaces $\mathcal{M}_{p}(\ren)$.

\begin{theorem}\label{Marcink}
Let  $p>\max\left\{1,\widetilde{p}\right\}$, choose the data $u_{0}\in \mathcal{M}_{p}(\ren)$, $u_{0}\geq0$ and let $u$ be the solution to
problem
\eqref{eqcauchy.3}. Then for all
$t>0$ we have $u(t)\in L^{\infty}(\ren)$ and
\begin{equation}
u(x,t)\leq C \,\|u_0\|_{\mathcal{M}_{p}}^{\sigma\beta_{p}}\,t^{-\alpha_{p}}\label{M_p-L}
\end{equation}
for a positive constant $C=C(m,N,p)$, the exponents $\alpha_p$, $\beta_p$ being defined in \eqref{expLorentz}.
\end{theorem}

\begin{proof}
As $u_{0}\in \mathcal{M}_{p}(\ren)$ setting $M=\|u_{0}\|_{\mathcal{M}_{p}}$,  by
Theorem \ref{Marthe} we have
\begin{equation*}
\|u(t)\|_{\infty}\leq\|\hat{U}_{p}\|_{\infty}=t^{-\alpha_{p}}M^{\sigma\beta_{p}}F_{p}(0)
\end{equation*}
and the result follows.
\end{proof}

\noindent  {\bf Extensions.} {\bf 1)} According to what we have seen in Section 2.2, for all $f\in L^{p}(\ren)$ we have
\[
\|f\|_{\mathcal{M}_{p}}\leq\|f\|_{L^{p}}
\]
by Theorem \ref{Marcink} the $L^{p}$-$L^{\infty}$ smoothing effect holds for the problem \eqref{eqcauchy.3}, in the sense of replacing the norm
$\|\cdot\|_{\mathcal{M}_{p}}$ with the
$L^{p}$ norm in \eqref{M_p-L}, for all $u_{0}\in L^{p}(\ren)$, for some $p>1$. This result has its
own interest, because for data in $L^p$ with $p>1$ we cannot find any ``worst case function'' that $u$ can be compared with. Besides, we may think to
choose the
data $u_{0}$ in a general Lorentz space $L^{p,q}(\ren)$ for $p,\,q\geq1$ and use the embedding \eqref{embed}, namely the inequality (not optimal again)
\[
\|f\|_{\mathcal{M}{p}}\leq C\|f\|_{L^{p,q}}\quad \forall f\in L^{p,q}(\ren)
\]
in order to obtain by Theorem \ref{Marcink} a more general $L^{p,q}$-$L^{\infty}$ smoothing effect. Of course, also in this case the constants are not
optimal.\\
We remind that the critical
exponent
\[
m_c=\frac{N-\sigma}{N}
\]
plays an important role in the existence of solutions. Indeed, according to \cite[Theorem 8.3]{pqrv2}, we have the existence of a unique \emph{strong} solution
both in
the \emph{supercritical regime} $m>m_{c}$ with $u_{0}\in L^{1}(\ren)$ and in the \emph{subcritical regime} $m\leq m_{c}$ and $u_{0}\in L^{1}(\ren)\cap
L^{p}(\ren)$ with $p>\widetilde{p}$, being $\widetilde{p}$ the critical exponent \eqref{critcexp}.

\medskip

\noindent {\bf 2)} With the same assumption on the nonlinearity $A$ and on the data $u_{0}$ given in Remark \ref{extgenA}, estimate \eqref{estimgenA} and the
construction of the selfsimilar solution $\hat{U}_{p}(x,t)$ provided in \eqref{selfMarc}  yield
\[
u(x,t)\leq C \,\|u_0\|_{\mathcal{M}_{p}}^{\sigma\beta_{p}}\,t^{-\alpha_{p}}
\]
where $C$ is a constant depending on $a,\,m,\,N,\,p$. This result generalizes the smoothing effect contained in \cite{pqrv4}.

%%%%%%%%%%%%%%%%%%%%%%%%%%%%%%%%%%%%%%%%%%%%%%%%%%%%%%%%%%%%%%
\section{Optimal estimates in the subcritical range. Extinction}

We now turn to the problem of finding optimal estimates  in the so-called \emph{subcritical FFD range,}
$m< m_c.$ A main qualitative feature of the theory in this range is the existence of solutions that vanish completely (\emph{i.e.} they extinguish) in finite
time. We want
to use the comparison theorem \ref{Marthe} to provide interesting estimates not only for the size of the solutions to the FFD equation \eqref{FPME} before
extinction, but also for the possible extinction time of such solutions. A special role will be played by the critical exponent $\widetilde{p}=N(1-m)/\sigma$
(cf. \eqref{critcexp}) which  is
larger than 1 for $m<m_c$.

Let us recall some known facts for the standard fast diffusion equation $u_t=\Delta u^m$. A classical result says that for exponents $0<m<(N-2)/N$ in dimension
$N\ge 3$, initial
data that are small enough produce weak solutions that vanish completely
after a finite time $T=T(u_0)>0$. The solutions are smooth, everywhere positive for $0<t<T$. Symmetrization techniques are used in \cite{JLVSmoothing}
to estimate the extinction time in terms of the Marcinkiewicz norm, following the idea that we have explained in the previous section.

In the case of fractional diffusion, $0<\sigma<2$, it has been proved in \cite{pqrv2} that a similar phenomenon of finite time extinction happens for
\begin{equation}0<m<m_{c}\label{range},
\end{equation}
 if the initial data belong to $L^{\widetilde{p}}(\ren)$.

%%%%%%%%%%%%%%%%%%%%%%%%%%%%%%%%%%%%%%%%%%%%%%%%%%%%%%%%%%%%%%%%%%%%%%%
\subsection{Explicit extinction solution}

In the range \eqref{range} of $m$ it is possible to construct an explicit singular solution to
\eqref{FPME}, vanishing at a finite time $T>0$ and belonging to the
Marcinkiewicz space $\mathcal{M}_{\widetilde{p}}\,(\ren)$. It is done as follows:  we look for such a solution as a function having the form
\[
U=c(t)|x|^{-\alpha}
\]
for some $\alpha>0$ and $c(t)>0$; then $U$  satisfies equation \eqref{FPME} if and only if
$\alpha=\sigma/(1-m)$, and
\begin{equation}
c^{\prime}(t)+\kappa(\alpha\,m)\,c^{m}(t)=0\label{eqc}.
\end{equation}
This can be easily derived from the explicit form of the fractional Laplacian $(-\Delta)^{\sigma/2}$ of the function $f=|x|^{-\alpha}$
(see
\cite{vazBaren}):
\[
(-\Delta)^{\sigma/2}|x|^{-\alpha}=\kappa(\alpha)|x|^{-\alpha-\sigma}
\]
where
\begin{equation}
\kappa(\alpha)=2^{\sigma}\frac{\Gamma((N-\alpha)/2)\,\Gamma((\alpha+\sigma)/2)}{\Gamma((N-\alpha-\sigma)/2)\,\Gamma(\alpha/2)}.\label{kappa}
\end{equation}
Checking the sign of all the terms in the expression of $\kappa(\alpha m)$ one realizes that condition \eqref{range} assures $\kappa(\alpha
m)>0$,
hence solving \eqref{eqc} with the initial condition $c(T)=0$ we find
\[
c(t)=\left[(1-m)\kappa\left(\frac{m\sigma}{1-m}\right)\right]^{\frac{1}{1-m}}(T-t)^{\frac{1}{1-m}}.
\]
Therefore, the explicit solution to \eqref{FPME} vanishing at finite time $T>0$ has the form
\begin{equation}\label{solvanish}
U(x,t;T)=\left[(1-m)\kappa\left(\frac{m\sigma}{1-m}\right)\right]^{\frac{1}{1-m}}
(T-t)^{\frac{1}{1-m}}|x|^{-\frac{\sigma}{1-m}}.
\end{equation}
We notice that $U(\cdot,\cdot;T)$ is actually a \emph{very weak} solution to problem \eqref{eqcauchy.3}, with the following initial data in the Marcinkiewicz space
$\mathcal{M}_{\widetilde{p}}\,(\ren)$
\[
u_{0}(x)=C_1\,T^{\frac{1}{1-m}}|x|^{-\frac{\sigma}{1-m}},
\qquad C_1=\left[(1-m)\kappa\left(\frac{m\sigma}{1-m}\right)\right]^{\frac{1}{1-m}}
\]
For the reader's convenience, we recall here the definition of very weak solution used in \cite{BV2012}.

\begin{definition}
A function $u$ is called a very weak solution to problem \eqref{eqcauchy.3} with data $u_{0}\in L^{1}_{loc}(\ren)$ if\\
\noindent $ \bullet$ we have
\begin{equation}u\in C([0,\infty);L^{1}_{loc}(\R^{N}),\,u^{m}\in L^{1}_{loc}((0,\infty);L^{1}(\ren,(1+|x|)^{-(N+\sigma)}dx));\label{funcspac}
\end{equation}\\
\noindent $ \bullet$ the following identity holds
\begin{equation}
\int_{0}^{\infty}\int_{\ren}u\frac{\partial\psi}{\partial t}\,dx\,dt-\int_{0}^{\infty}\int_{\ren}u^{m}(-\Delta)^{\sigma/2}\psi\, dx,dt=0\
\label{idveryweak}
\end{equation}
for all the test functions $\psi\in C_{0}^{\infty}(\ren\times[0,T])$, $T>0$;\medskip\\
\noindent $ \bullet$ $u(\cdot,0)=u_{0}\in L^{1}_{loc}(\ren)$.
\end{definition}
In order to justify that $U(\cdot,\cdot\,;T)=U_{T}$ agrees with such definition, we first observe that \eqref{funcspac} is satisfied because $U$
is
integrable around the origin for all times $t$. Besides, by the construction of $U_{T}$ and condition \eqref{range} we get
\[
\frac{\partial U_{T}}{\partial t},\,(-\Delta)^{\sigma/2} U_{T}\in L^{1}_{loc}(\ren\times(0,T))
\]
then for all the test functions $\psi$ we find
\begin{align*}
&0=\int_{0}^{\infty}\int_{\ren}\psi\frac{\partial U_{T}}{\partial t}+\int_{0}^{\infty}\int_{\ren}\psi(-\Delta)^{\sigma/2}U_{T}^{m}\,
dx,dt\\
&=-\int_{0}^{\infty}\int_{\ren}U_{T}\frac{\partial\psi}{\partial
t}\,dx\,dt+\int_{0}^{\infty}\int_{\ren}U_{T}^{m}(-\Delta)^{\sigma/2}\psi\,
dx\,dt.
\end{align*}

%%%%%%%%%%%%%%%%%%%%%%%%%%%%%%%%%%%%%%%%%%%%%%%%%%%%%%%%%%%%%%%%%%%%%%%
\subsection{Optimal estimates for extinction times}

We can use Theorem \ref{Marthe} to compare a weak solution $u$ to problem \eqref{eqcauchy.3} with the explicit solution \eqref{solvanish}, in order to get an
estimate of the extinction time, exactly like in the  theory involving the classical Laplacian.

\begin{theorem}\label{extinctiontimetheo} Let $\sigma\in (0,2)$ and let $u$ be the solution to \eqref{eqcauchy.3} with nonnegative data $u_{0}\in\mathcal{
M}_{\widetilde{p}}(\ren)$
where we assume that  $0<m<m_c$.
Then $u$ vanishes in a finite time $T(u_0)$ and
\begin{equation}
T(u_0)\leq d(\sigma,m)\|u_{0}\|_{\mathcal{M}_{\widetilde{p}}(\ren)}^{1-m}\label{extinctime}
\end{equation}
where
\[
d(\sigma,m)=\left((1-m)\kappa\left(\frac{m\sigma}{1-m}\right)\right)^{-1},
\]
and the function $\kappa(\alpha)$ is given in \eqref{kappa}.
Moreover, for all $0<t<T$ we have $u(t)\in \mathcal{M}_{\widetilde{p}}(\ren)$ and $u^{\#}(t)\prec U(t;T)$.
\end{theorem}

\begin{proof}
Set
\[
\rho=\left[(1-m)\kappa\left(\frac{m\sigma}{1-m}\right)\right]^{\frac{1}{1-m}}
\]
and
\[
T=\left(\frac{\|u_{0}\|_{M_{\widetilde{p}}}}{\rho}\right)^{1-m}
\]
and consider the solution \eqref{solvanish} vanishing at time $T$. Then the solution $U(x,t;T)$ gets initial data
\[
\|u_{0}\|_{M_{\widetilde{p}}}\,|x|^{-\frac{\sigma}{1-m}}.
\]
Therefore, by Theorem \ref{Marthe}
\[
U(\cdot,t;T)\succ u^{\#}(\cdot,t)\,,
\]
thus for all $t\geq T$,
\[
\|u(\cdot,t)\|_{\infty}\leq \|U(\cdot,t;T)\|_{\infty}=0.
\]
Recall now that the vanishing time of $u$ is  $ T(u_{0})=\inf\left\{t\geq 0:\|u(\cdot,t)\|_{\infty}=0\right\}$.
\end{proof}

\begin{remark}[Extinction time for fractional equations with general nonlinearity] \emph{The comparison of concentrations and the estimate of the extinction
time are still valid, and represent a new result, for the solutions of an equation of the form
\eqref{gennonl} when $A$ is not a power function, but $A'(u)\ge au^{m-1}$ for some $a>0$ and all $u>0$, where $m\in (0,m_c)$. Indeed, by \cite{pqrv4} we get the existence of a unique weak solution $u$ to \eqref{gennonl} with data $u_{0}\in L^{\widetilde{p}}(\ren)$. Then arguing as in Remark \ref{extgenA} we find the estimate \eqref{estimgenA}, thus by Proposition \ref{extinctiontimetheo} we find that the solution $u$ vanishes with finite time
\begin{equation}
T_{1}(u_{0})\le \frac{m}{a}d(\sigma,m)\|u_{0}\|_{\mathcal{M}_{\widetilde{p}}(\ren)}^{1-m}\,.
\end{equation}
where the end of the expression is taken from \eqref{extinctime}.}

\end{remark}

%%%%%%%%%%%%%%%%%%%%%%%%%%%%%%%%%%%%%%%%%%%%%%%%%%%%

\section{Extensions and open problems}

\noindent $\bullet$  {\bf More general operators.} The above estimates could be extended to the solutions of more general versions of the fractional Laplacian operator, in the same way that the standard symmetrization applies to elliptic equations with coefficients. We will not give further details since it would be convenient to work out the consequences for the existence theory, and this maybe deserves a proper space.

\noindent $\bullet$  {\bf Open problem.} We do not know how to perform the parabolic comparison in the case of more general $A$, which is not assumed to be concave or comparable to a concave function. In particular, we do not know how to obtain best constants for the FPME with $m>1$. Are they still given by
the constant of the Barenblatt solution? Any information on this issue would be welcome.

\

%%%%%%%%%%%%%%%%%%%%%%%%%%%%%%%%%%%%%%%%%%%%%%%%%%%%%%%%%%%%%%%%%%%%%%%%
%%%%%%%%%%%%%%%%%%%%%%%%%%%%%%%%%%%%%%%%%%%%%%%%%%%%%%%%%%%%%%%%%%%

\noindent {\large\bf Acknowledgments}

\noindent Both authors partially supported by the Spanish Project MTM2011-24696. We thank Matteo Bonforte for his interest in this work and some suggestions.

\medskip
%%%%%%%%%%%%%%%%%%%%%%%%%%%%%%%%%%%%%%%%%%%%%%%

{\small

%%%%%%%%%%%%%%%%%%%%

}

\

\noindent { \sc Addresses}

\noindent Juan Luis Vázquez, Dpto. de Matem\'aticas, Univ. Aut\'onoma de Madrid, \\
28049 Madrid, Spain. \quad   E-mail: {\tt juanluis.vazquez@uam.es}

\noindent Bruno Volzone, Dipartimento di Ingegneria,\\ Università degli Studi di Napoli ``Parthenope'', 80143 Napoli, Italy. \\E-mail: {\tt bruno.volzone@uniparthenope.it}


\begin{thebibliography}{99}


\bibitem{Applebaum} {\sc D. Applebaum}. {\lq\lq L\'{e}vy processes and stochastic calculus''}.
Second edition. Cambridge Studies in Advanced Mathematics, 116. Cambridge
University Press, Cambridge, 2009.

\bibitem{Aubin76}{\sc T. Aubin.} {\sl Problèmes isopérimétriques et espaces de Sobolev.} J. Diff. Geom. {\bf 11} (1976), 573--598.


\bibitem{Bandle}{\sc C. Bandle}.
{\sl Isoperimetric inequalities and applications.}
Monographs and Studies in Mathematics, 7. Pitman
(Advanced Publishing Program), Boston, Mass.-London, 1980.

\bibitem{Band2}{\sc C. Bandle}.
{\sl On symmetrizations in parabolic equations},
J. Analyse Math.  {\bf 30}, (1976), 98--112.

\bibitem{Bennett-Sharpley}
{\sc C.~Bennett and R.~Sharpley}. {\em Interpolation of operators}, vol.~129 of
  Pure and Applied Mathematics, Academic Press Inc., Boston, MA, 1988.

\bibitem{Bertoin} {\sc J. Bertoin.}
\lq\lq L\'{e}vy processes''. Cambridge Tracts in Mathematics,
121. Cambridge University Press, Cambridge, 1996. ISBN: 0-521-56243-0.

\bibitem{BV2012}
{\sc  M. Bonforte and J.~L.~V{\'a}zquez}.
{\sl Quantitative local and global  a priori estimates for fractional nonlinear diffusion equations.} Advances in Mathematics, to appear. In  arXiv:1210.2594v2
[math.AP].


\bibitem{Davies1}{\sc E. B. Davies}. { ``Heat kernels and spectral theory'',} Cambridge Tracts in Mathematics, 92. Cambridge University Press, Cambridge, 1990.
    x+197 pp. ISBN: 0-521-40997-7

\bibitem{BV} {\sc G. Di Blasio and B. Volzone}.
{\sl Comparison and regularity results for the fractional Laplacian via symmetrization methods},
J. Differential Equations  {\bf 253}, 9  (2012), 2593--2615.

\bibitem{pqrv} {\sc  A. de Pablo, F. Quir\'os, A. Rodr\'{\i}guez, and J.~L. V\'azquez}.
  {\sl A fractional porous medium equation.} Adv. Math. {\bf 226} (2011), no.~2, 1378--1409.

\bibitem{pqrv2} {\sc  A. de Pablo, F. Quir\'os, A. Rodr\'{\i}guez, and J.~L. V\'azquez}.
    {\sl A general fractional porous medium equation.}
    Comm. Pure Appl. Math. {\bf 65} (2012), no.~9, 1242--1284.

\bibitem{pqrv4} {\sc  A. de Pablo, F. Quir\'os, A. Rodr\'{\i}guez, and J.~L. V\'azquez}.
{\sl Smooth solutions for nonlinear fractional diffusion equations}, in preparation.

\bibitem{MR0046395}
{\sc G.~H. Hardy, J.~E. Littlewood, and G.~P{\'o}lya}. {``Inequalities''},
  Cambridge University Press, 1952,  2d ed.


\bibitem{Lieb83} {\sc E. H. Lieb. } {\sl Sharp Constants in the Hardy-Littlewood-Sobolev and  Related Inequalities}. Annals of Math., {\bf 118}, (1983),
    349--374



\bibitem{Maz}
{\sc V.~G. Maz'ja}. {\sl Weak solutions of the Dirichlet and Neumann problems},  Trudy Moskov. Mat. Ob\u{s}u{c}. {\bf 20} (1969),   137--172.
 in  Russian


\bibitem{PS1951}
{\sc G. P\'olya and C. Szeg\"o}. {\sl ``Isoperimetric inequalities in Mathematical Physics'',}
Annals of Mathematics Studies, vol. 27, Princeton University Press, Princeton, N.J., 1951.


\bibitem{Talenti1} {\sc G. Talenti.} {\sl Elliptic equations and rearrangements,} Ann. Scuola Norm. Sup. (4) 3 (1976), 697--718.

\bibitem{Talenti2} {\sc G. Talenti.} {\sl Best constant in Sobolev inequality,} Ann. Mat. Pura Appl. (4) 110 (1976), 353--372.

\bibitem{Talenti3} {\sc G. Talenti.} {\sl Nonlinear elliptic equations,
rearrangements of functions and Orlicz spaces,}
Annal. Mat. Pura Appl. 4, 120 (1979), 159--184.


\bibitem{Valdinoc} {\sc E. Valdinoci.} {\sl From the long jump random walk to the fractional Laplacian}, Bol. Soc. Esp. Mat. Apl. {\bf 49}
    (2009),
    33--44.

\bibitem{Vsym82} {\sc J.~L. Vazquez}.
{\sl Sym\'etrisation pour $u_t=\Delta\varphi(u)$ et applications,}
C. R. Acad. Sc. Paris {\bf  295} (1982), pp. 71--74.


\bibitem{JLVSmoothing} {\sc J.~L. V{\'a}zquez.} \lq\lq Smoothing And Decay Estimates
For Nonlinear Diffusion Equations. Equations Of Porous Medium
Type''. Oxford Lecture Series in Mathematics and its Applications,
33. Oxford University Press, Oxford, 2006. % MR2282669 (2007K:35008)

\bibitem{vazBaren} {\sc J.~L.V\'{a}zquez}.
{\sl Barenblatt  solutions and asymptotic behaviour for a  nonlinear fractional heat equation of porous medium type, } preprint in arXiv:1205.6332v2.

\bibitem{VazVol} {\sc J.~L. V\'{a}zquez and B.~Volzone}.
{\sl Symmetrization for Linear and Nonlinear Fractional  Parabolic Equations of Porous Medium Type,} Journal Math. Pures Appli., to appear; \normalcolor
arXiv:1303.2970.

 \bibitem{Wein62} {\sc H. Weinberger}. {\sl Symmetrization in uniformly elliptic problems}, Studies in
Math. Anal., Stanford Univ. Press, 1962, pp. 424--428.


\end{thebibliography}
\end{document}